\documentclass{amsart} 

\newtheorem{theorem}{Theorem}[section]
\newtheorem{proposition}{Proposition}[section]
\newtheorem{lemma}{Lemma}[section]

\theoremstyle{definition}

\newtheorem{remark}{Remark}[section]

\usepackage{amscd}
\usepackage{amsthm}
\usepackage{amsfonts}
\usepackage{amsbsy}
\usepackage{amsmath}
\usepackage{amssymb}
\usepackage{color}	
\usepackage{enumerate}
\usepackage{epsfig}
\usepackage[mathscr]{eucal}
\usepackage{fancyhdr}
\usepackage{mathrsfs}
\usepackage{graphicx}
\usepackage{hyperref}
\hypersetup{colorlinks=false, urlcolor=black}
\usepackage{latexsym}
\usepackage{listings}
\usepackage{nomencl}
\usepackage[numbers]{natbib}
\usepackage{psfrag}
\usepackage{setspace}
\usepackage{tikz}
\usepackage{url}
\usepackage{undertilde}
\usepackage{verbatim}
\usepackage{graphics,epstopdf}
\usepackage[all]{xy}
\newcommand{\Z}{\mathbb{Z}}
\newcommand{\Q}{\mathbb{Q}}

\newcommand{\K}{\mathbb{K}}
\newcommand{\C}{\mathbb{C}}
\newcommand{\F}{\mathbb{F}}
\newcommand{\E}{\mathbb{E}}

\begin{document}

\title[Arithmetic matrices for number fields II]{Arithmetic matrices for number fields II: Parametrization of rings by binary forms}

\author[Samuel A. Hambleton]{Samuel A. Hambleton}

\address{School of Mathematics and Physics, The University of Queensland, St. Lucia, Queensland, Australia 4072}

\email{sah@maths.uq.edu.au}


\date{August 29, 2018.}


\begin{abstract}
We show that binary forms of degree n less than seven parameterize rings, generalizing a result of Levi on binary cubic forms parameterizing cubic rings, which can be related to results of Bhargava. The question of whether or not a binary form of degree n parameterizes a ring of rank $n$ depends on a symbolic calculation with matrices. 
\end{abstract}

\maketitle

\section{Introduction}\label{intro}

The parametrization of cubic rings was investigated by Delone and Faddeev \cite[pp. 106]{DnF1}, who showed that two binary cubic forms that belong to the same $\text{GL}_2(\Z )$-class of binary cubic forms produce the same ring up to isomorphism. That is, if $\mathcal{C}(x, y) = (a,b,c,d)$ can be transformed into $\overline{\mathcal{C}}(x, y) = \left( \overline{a}, \overline{b}, \overline{c}, \overline{d} \right)$ by replacing $\left(
\begin{array}{c}
 x \\
 y \\
\end{array}
\right)$ with $M \left(
\begin{array}{c}
 x \\
 y \\
\end{array}
\right)$, where $M \in \text{GL}_2(\Z )$, and we write $\mathcal{C} \circ M = \overline{\mathcal{C}}$ when this is the case, then one obtains two isomorphic cubic rings from $\mathcal{C} $ and $\overline{\mathcal{C}}$. Delone remarks \cite[pp. xiii]{DnF1} that this result is due to Friedrich Wilhelm Levi \cite{Levi2} also developed by Delone \cite[pp. 106]{DnF1}. The result has been generalized by Bhargava \cite{Bhargone3, Bhargone4}, who introduced the concept of resolvent rings to deal with the enormous calculations which emerge in considering binary forms of degree greater than $3$. In these articles pairs or quadruples of forms were used to parameterize rings of ranks $4$ and $5$. The aim of this article is to introduce a technique to actually perform the large symbolic computations that come about from parameterizing rings of rank $n$ by binary forms of degree $n$ with the hope that the method works in general. We prove the following result. 
\begin{proposition}\label{mainparam}
Let $\mathcal{B}_n = \left( a_1, a_2, \dots , a_{n+1} \right)$ be a binary form of degree $n$, where $3 \leq n \leq 6$, so that $a_1, a_2, \dots , a_{n+1}$ are any rational integers such that $a_1 a_{n+1} \not= 0$ and
\begin{equation*}
\mathcal{B}_n(x, y) = \sum_{k = 1}^{n+1} a_k x^{n+1-k} y^{k-1} 
\end{equation*}
is irreducible in $\Q [x, y]$. Let 
\begin{equation*}
M = \left(
\begin{array}{cc}
 p & q \\
 r & s \\
\end{array}
\right) \in \text{GL}_2(\Z ) . 
\end{equation*}
Then the binary form $\mathcal{B}_n \circ M$ obtained by replacing $\left(
\begin{array}{c}
 x \\
 y \\
\end{array}
\right)$ in $\mathcal{B}_{n}(x, y)$ with $M \left(
\begin{array}{c}
 x \\
 y \\
\end{array}
\right)$ parameterizes a ring $\mathcal{R}_{M} = \left[ 1 , \psi_1, \psi_2, \dots , \psi_{n-1} \right] $ that is isomorphic to the ring $\mathcal{R}$ parameterized by $\mathcal{B}_n$, $\mathcal{R} = \left[ 1 , \phi_1, \phi_2, \dots , \phi_{n-1} \right] $, where the $\phi_j$ are of the form
\begin{align*}
 \phi_j & = \sum_{k = 1}^{j} a_k \zeta^{j+1-k} \ (j > 1) , & \phi_0 & = 1,
\end{align*}
and $\zeta $ is a root of $\mathcal{B}_n(x, 1)$. 
\end{proposition}

In using the term parametrization, we emphasize that the integers $a_1, a_2, \dots , a_{n+1}$ give rise to a ring, and if two binary forms belong to the same $\text{GL}_2(\Z)$-class, then the binary forms give rise to the same ring. Also, there is no reason to believe that Proposition \ref{mainparam} does not hold when $n > 6$. In fact, the author conjectures that it holds for all positive integers greater than $2$. Further, we show that verification of the result for any $n > 6$ can be done with a calculation. However, before we prove Proposition \ref{mainparam}, we first need to introduce the appropriate tools.   

In \cite{hamblarithnf} the following proposition was proved. It generalizes the arithmetic matrices introduced in \cite{hamblbrahm, cfg} and facilitates basic arithmetic in the ring of integers of a number field by doing so with $n$ by $n$ matrices.  The entries of these matrices are taken to belong to $\Z \left[ x_0, x_1, x_2, \dots , x_n \right]$, where the $x_0, x_1, x_2, \dots , x_n$ are coefficients of generators of a particular integral basis for the ring of integers of the number field it refers to. Very often the $x_0, x_1, x_2, \dots , x_n$ will be rational integers. 

\begin{proposition}\label{mainres}
Let $\E = \Q (\zeta )$ be a number field of degree $n$ over $\Q$, where $\zeta $ is a root of the irreducible polynomial $f(x) = a_1 x^n + a_2 x^{n-1} + \dots + a_{n+1} \in \Z [x]$ with roots $\zeta_0, \zeta_1, \dots , \zeta_{n-1}$; $\zeta_0 = \zeta $. Let 
\begin{align}\label{intbasn}
\mathcal{O}_{\E } & = \left[ \rho_0 , \rho_1 , \rho_2 , \dots , \rho_{n-1} \right] , & \rho_j & = \sum_{k = 1}^{j} a_k \zeta^{j+1-k} \ (j > 1) , & \rho_0 & = 1.
\end{align}
be an integral basis for the ring of integers $\mathcal{O}_{\E }$ of $\E $. Let $\kappa_{0}, \kappa_1, \dots , \kappa_{n-1}$ be the embeddings of $\E$ in $\C$ so that $\kappa_{t}\left( \zeta_{u} \right) = \zeta_{v}$, $v \equiv t + u \pmod{n}$. Let $\Gamma_{\E} = \left[ \kappa_{i-1} \left( \rho_{j-1} \right) \right]$ and let $\Theta_{\E }^{(\alpha )} = \left[ \delta_{ij} \kappa_{i-1} \left( \alpha \right) \right] $, where $\delta_{ij}$ is the Kronecker-delta symbol, ($1$ if $i=j$ and $0$ otherwise) so that $\Theta_{\E }^{(\alpha )}$ is a diagonal matrix. Let $N_{\E}^{(\alpha )} = \Gamma_{\E}^{-1} \Theta_{\E }^{(\alpha )} \Gamma_{\E}$, and let $M_{\E}^{(\alpha )} = \left[ a_{ij} \right]$, where the entries $a_{ij}$ of $M_{\E}^{(\alpha )}$ are given by 
\begin{eqnarray}
\label{defaijone} a_{11} & = & x_0 , \\
\label{defaijtwo} a_{i1} & = & x_{i-1} , \ \text{for} \ i > 1 , \\
\label{defaijthreee} a_{1j} & = & - a_{n+1} \sum_{k = 1}^{j-1} a_k x_{k+n-j}, \ \text{for} \ j > 1 , \\
\label{defaijfour} a_{ij} & = &  \sum_{k = 1}^{j-1} a_k x_{k+i-j-1} \ \text{for} \ i > j > 1 , \\
\label{defaijfive} a_{ij} & = & \delta_{ij} x_0 - \sum_{k = j}^{m} a_k x_{k+i-j-1} \ \text{for} \ j \geq i > 1 , \\
\nonumber m & = & \min (n - i + j, n + 1) .
\end{eqnarray} 
Then the following properties of the matrix $N_{\E}^{(\alpha )}$ hold for $\alpha , \beta \in \mathcal{R} \left[ \rho_0, \rho_1, \dots , \rho_{n-1} \right]$, where $\mathcal{R}$ is a commutative ring with $x_0, x_1, \dots , x_{n-1} \in \mathcal{R}$, $\alpha = \sum_{j=0}^{n-1} x_j \rho_j$. 
\begin{enumerate}
\item $N_{\E}^{(\alpha )} = M_{\E}^{(\alpha )}$.
\item $N_{\E}^{(\alpha )} + N_{\E}^{(\beta )} = N_{\E}^{(\alpha + \beta )}$.
\item $N_{\E}^{(\alpha )} N_{\E}^{(\beta )} = N_{\E}^{(\beta )} N_{\E}^{(\alpha )} = N_{\E}^{(\alpha \beta )}$.
\item The trace of the matrix $N_{\E}^{(\alpha )}$ is equal to the trace of $\alpha $.
\item The determinant of the matrix $N_{\E}^{(\alpha )}$ is equal to the norm of $\alpha $.
\item $N_{\E}^{(1 / \alpha )} = \left( N_{\E}^{(\alpha )} \right)^{-1}$. 
\item $N_{\E}^{(\alpha )}$ has entries in $\Z$ if and only if $\alpha \in \mathcal{O}_{\E }$. 
\end{enumerate}
\end{proposition}

This proposition can be used to more efficiently multiply algebraic integers when several multiplications are required, such as might be necessary when working with the integral bases of ideals, as pointed out in \cite{hamblarithnf}.  

In order to work with binary forms that do not necessarily have discriminant equal to that of a number field of degree $n$, we must define the matrix 
\begin{equation*}
N_{\mathcal{O}}^{(\alpha )} = \Gamma_{\mathcal{O}}^{-1} \Theta_{\mathcal{O}}^{(\alpha )} \Gamma_{\mathcal{O}}, 
\end{equation*}
where $\Gamma_{\mathcal{O}} = \left[ \kappa_{i-1} \left( \phi_{j-1} \right) \right]$ and $\mathcal{O} = \left[ 1, \phi_1, \phi_2 , \dots , \phi_{n-1} \right]$ is an order of $\E = \Q ( \zeta ) $. In other words we permit the $a_1, a_2, \dots , a_{n+1}$ to be any rational integers such that $a_1 a_{n+1} $ is non-zero and $\mathcal{B}_n(x, y)$ is irreducible in $\Q [x, y]$. We must assume that the generators of $\mathcal{O}$ have the form
\begin{align}\label{genform}
 \phi_j & = \sum_{k = 1}^{j} a_k \zeta^{j+1-k} \ (j > 1) , & \phi_0 & = 1.
\end{align}
Hence $N_{\mathcal{O}}^{(\alpha )}$ gives a matrix representation of the arithmetic of the order $\mathcal{O}$. All of the identities in Proposition \ref{mainres} still hold when the discriminant of the binary form $\mathcal{B}_n = \left( a_1, a_2, \dots , a_{n+1} \right)$ is not the discriminant of a field of degree $n$ over $\Q$; however the order $\mathcal{O}$ in which $N_{\mathcal{O}}^{(\alpha )}$ performs arithmetic will not be the maximal order of the field generated by a root $\zeta $ of $\mathcal{B}_n(x, 1) $.

\section{Cubic rings}

Now we are able to illustrate how the arithmetic matrices are helpful in understanding the Levi correspondence. Let $n = 3$ and let $\mathcal{C} = (a,b,c,d)$ be a binary cubic form so that $a, b, c, d$ are any four rational integers such that $ad \not= 0$ and $\mathcal{C}(x, y)$ is irreducible over $\Q$. The formulas given in Proposition \ref{mainres} above provide the arithmetic matrix for the order $\mathcal{O}$ of the cubic field $\K = \Q (\zeta )$ given by
\begin{align}\label{arithcubic}
N_{\mathcal{O}}^{(\alpha )} & = \left(
\begin{array}{ccc}
 u & -a d y & -d (a x+b y) \\
 x & u-b x-c y & -c x-d y \\
 y & a x & u-c y \\
\end{array}
\right) , & \alpha & = u + x \phi_1 + y \phi_2 , 
\end{align}
where $\phi_1 = a \zeta $ and $\phi_2 = a \zeta^2 + b \zeta $. Since $(a,b,c,d)$ looks better than $\left( a_1 , a_2 , a_3 , a_4 \right)$, we have replaced these accordingly to produce $N_{\mathcal{O}}^{(\alpha )}$ from the formulas in Proposition \ref{mainres}. When $\mathcal{C}$ does not have discriminant equal to that of a cubic field, then $\mathcal{A} = \{ 1 , \phi_1 , \phi_2 \}$ is not an integral basis for the ring of integers; instead, $\mathcal{A}$ simply generates an order $\mathcal{O}$ of $\K$. 

We have 
\begin{align*}
 N_{\mathcal{O}}^{\left( \phi_1 \right) } & = \left(
\begin{array}{ccc}
 0 & 0 & -a d \\
 1 & -b & -c \\
 0 & a & 0 \\
\end{array}
\right) ,  & N_{\mathcal{O}}^{\left( \phi_2 \right) } & = \left(
\begin{array}{ccc}
 0 & -a d & -b d \\
 0 & -c & -d \\
 1 & 0 & -c \\
\end{array}
\right) .
\end{align*} 
We calculate the matrices
\begin{align*}
\left( N_{\mathcal{O}}^{\left( \phi_1 \right) } \right)^2 & = \left(
\begin{array}{ccc}
 0 & -a^2 d & 0 \\
 -b & b^2-a c & b c-a d \\
 a & -a b & -a c \\
\end{array}
\right) , \\ 
 N_{\mathcal{O}}^{\left( \phi_1 \right) } N_{\mathcal{O}}^{\left( \phi_2 \right) } & = \left(
\begin{array}{ccc}
 -a d & 0 & a c d \\
 -c & b c-a d & c^2 \\
 0 & -a c & -a d \\
\end{array}
\right) , \\
\left( N_{\mathcal{O}}^{\left( \phi_2 \right) } \right)^2 & = \left(
\begin{array}{ccc}
 -b d & a c d & a d^2+b c d \\
 -d & c^2 & 2 c d \\
 -c & -a d & c^2-b d \\
\end{array}
\right) .
\end{align*}
By considering the first column of these matrices we obtain the following multiplication formulas for the ring with basis $\mathcal{A} = \left\{ 1, \phi_1, \phi_2 \right\}$.
\begin{eqnarray}
\label{cubrone} \phi_1^2 & = & - b \phi_1 + a \phi_2 , \\
\label{cubrtwo} \phi_1 \phi_2 & = & - a d - c \phi_1 , \\
\label{cubrthr} \phi_2^2 & = & - b d - d \phi_1 - c \phi_2 .
\end{eqnarray}
Alternatively, we could obtain the columns of interest and hence the coefficients of $\phi_0, \phi_1, \phi_2$ in the expressions \eqref{cubrone}, \eqref{cubrtwo} and \eqref{cubrthr} by computing 
\small
\begin{align*}
N_{\mathcal{O}}^{\left( \phi_1 \right) } \left(
\begin{array}{c}
 0  \\
 1  \\
 0  \\
\end{array}
\right) & = \left(
\begin{array}{c}
 0  \\
 -b \\
 -c  \\
\end{array}
\right) , & N_{\mathcal{O}}^{\left( \phi_1 \right) } \left(
\begin{array}{c}
 0  \\
 0  \\
 1  \\
\end{array}
\right) & = \left(
\begin{array}{c}
 -a d  \\
 -c \\
 0  \\
\end{array}
\right) , & N_{\mathcal{O}}^{\left( \phi_2 \right) } \left(
\begin{array}{c}
 0  \\
 0  \\
 1  \\
\end{array}
\right) & = \left(
\begin{array}{c}
 -b d  \\
 -d \\
 -c  \\
\end{array}
\right) .
\end{align*}
\normalsize 
By letting $\phi = \phi_1$ and $\psi = \phi_2 + c$, we obtain the isomorphic ring generated by $\{ 1, \phi , \psi \}$, where 
\begin{eqnarray*}
\phi^2 & = & - a c - b \phi + a \psi , \\
\phi \psi & = & - a d , \\
\psi^2 & = & - b d - d \phi + c \psi .
\end{eqnarray*}
It is these multiplication formulas that Delone and Faddeev \cite[pp. 106]{DnF1}, Bhargava \cite{Bhargone0} and others exhibited.

Next we show that if $M \in \text{GL}_2(\Z )$, then $\mathcal{C} \circ M$ parameterizes a ring isomorphic to the ring parameterized by $\mathcal{C}$, the Levi correspondence \cite{Levi2}.
\begin{theorem}\label{cubicpar}
Let $\mathcal{C} = (a,b,c,d)$ be a binary cubic form so that $a,b,c,d$ are any rational integers such that $\mathcal{C}(x, y)$ is irreducible over $\Q $. Let $$M = \left(
\begin{array}{cc}
 p & q \\
 r & s \\
\end{array}
\right) \in \text{GL}_2(\Z ) .$$ Then the binary cubic form $\mathcal{C} \circ M$ parameterizes a cubic ring $\mathcal{R}_{M}$ that is isomorphic to the cubic ring $\mathcal{R}$ parameterized by $\mathcal{C}$. 
\end{theorem}

\begin{proof}
Let $\zeta $ be a root of $\mathcal{C}(x, 1)$. Then $\tau = \frac{s \zeta - q}{-r \zeta + p}$ is a root of $\overline{\mathcal{C}}(x, 1)$, where $\overline{\mathcal{C}} = \mathcal{C} \circ M$. Let $m = \det (M) \ (= \pm 1)$. We define a map 
\small 
\begin{align*}
\lambda  & : \mathcal{R}_{M} \longrightarrow \mathcal{R} , \\
\lambda & : u + x \psi_1 + y \psi_2 \longmapsto u + t_{12} x + t_{13} y + m (p x + q y) \phi_1 + m (r x + s y) \phi_2 ,
\end{align*}
\normalsize 
where
\begin{eqnarray*}
t_{12} & = & - \left( a q p^2 + b q r p + c q r^2 + d r^2 s \right) , \\
t_{13} & = & - \left( 2 a p q^2+b r q^2+b p s q+2 c r s q+2 d r s^2 \right) ,
\end{eqnarray*}
the generators are related by
\begin{align*}
\phi_1 & = a \zeta , & \phi_2 & = a \zeta^2 + b \zeta , \\ 
\psi_1 & = \overline{a} \tau , & \psi_2 & = \overline{a} \tau^2 + \overline{b} \tau ,
\end{align*}
and the form $\overline{\mathcal{C}} = \mathcal{C} \circ M$ has coefficients
\begin{align*}
\overline{a} & =  a p^3 + b p^2 r + c p r^2 + d r^3 , & & \\
\overline{b} & =  3 a p^2 q + b p^2 s + 2 b p q r + 2 c p r s + c q r^2 + 3 d r^2 s , & & \\
\overline{c} & =  3 a p q^2 + 2 b p q s + b q^2 r + c p s^2 + 2 c q r s + 3 d r s^2 , & & \\ 
\overline{d} & =  a q^3 + b q^2 s + c q s^2 + d s^3 . & & 
\end{align*}
To show that $\lambda$ is a ring homomorphism, let 
\begin{align*}
\alpha & = u_1 + x_1 \psi_1 + y_1 \psi_2 , & \beta & = u_2 + x_2 \psi_1 + y_2 \psi_2 ,
\end{align*}
Observe that 
\begin{eqnarray*}
\lambda (\alpha ) \lambda (\beta ) & = & \left( u_3 + x_3 \phi_1 + y_3 \phi_2 \right) \left( u_4 + x_4 \phi_1 + y_4 \phi_2 \right) , \\
                                   & = & u_5 + x_5 \phi_1 + y_5 \phi_2 , 
\end{eqnarray*}
where, using the arithmetic matrices we find that 
\begin{align*}
u_3 & =  u_1 + t_{12} x_1 + t_{13} y_1 ,  & u_4 & = u_2 + t_{12} x_2 + t_{13} y_2 , \\
x_3 & =  m \left( p x_1 + q y_1 \right) , & x_4 & = m \left( p x_2 + q y_2 \right) , \\
y_3 & =  m \left( r x_1 + s y_1 \right) , & y_4 & = m \left( r x_2 + s y_2 \right) , 
\end{align*}
and
\begin{eqnarray*}
u_5 & = & u_3 u_4 - a d x_4 y_3 - a d x_3 y_4 - b d y_3 y_4 , \\
x_5 & = & u_3 x_4 + u_4 x_3 - b x_3 x_4 - c x_3 y_4 - c x_4 y_3 - d y_3 y_4 , \\
y_5 & = & u_3 y_4 + u_4 y_3 + a x_3 x_4 - c y_3 y_4 .
\end{eqnarray*}
Now 
\begin{align*}
\alpha \beta  & = u_6 + x_6 \psi_1 + y_6 \psi_2 , & \lambda ( \alpha \beta ) & = u_7 + x_7 \phi_1 + y_7 \phi_2 ,
\end{align*} 
where 
\begin{eqnarray*}
u_6 & = & u_1 u_2 - \overline{a} \overline{d} x_2 y_1 - \overline{a} \overline{d} x_1 y_2 - \overline{b} \overline{d} y_1 y_2 , \\
x_6 & = & u_1 x_2 + u_2 x_1 - \overline{b} x_2 x_1 - \overline{c} x_1 y_2 - \overline{c} x_2 y_1 - \overline{d} y_1 y_2 , \\
y_6 & = & u_1 y_2 + u_2 y_1 + \overline{a} x_1 x_2 - \overline{c} y_1 y_2 , \\
u_7 & = & u_6 + t_{12} x_6 + t_{13} y_6 , \\
x_7 & = & m \left( p x_6 + q y_6 \right) , \\
y_7 & = & m \left( r x + s y \right) .
\end{eqnarray*}
Taking differences, we find that $u_7 - u_5 = x_7 - x_5 = y_7 - y_5  = 0$.

It follows that $\lambda (\alpha ) \lambda (\beta ) = \lambda (\alpha \beta )$. It is easy to show that $\lambda (\alpha ) + \lambda (\beta ) = \lambda (\alpha + \beta )$. To show that $\lambda $ is a ring isomorphism, assume that 
\begin{eqnarray*}
u + t_{12} x + t_{13} y & = & 0, \\
m (p x + q y) & = & 0 ,\\
m (r x + s y) & = & 0 .
\end{eqnarray*}
Then, since the matrix 
\begin{equation}\label{chbas}
T = \left(
\begin{array}{ccc}
 1 & t_{12} & t_{13} \\
 0 & m p  & m q  \\
 0 & m r  & m s  \\
\end{array}
\right) \in \text{SL}_3(\Z ) 
\end{equation}
is invertible, we must have $u = x = y = 0$. Thus the kernel of $\lambda$ as a group homomorphism under addition is trivial. If we assume that 
\begin{eqnarray*}
u + t_{12} x + t_{13} y & = & 1, \\
m (p x + q y) & = & 0 ,\\
m (r x + s y) & = & 0 , 
\end{eqnarray*}
Then we must have $u = 1$, $x = 0$, and $y = 0$. It follows that the kernel of $\lambda$ as a group homomorphism under multiplication is trivial. 

Finally, we must show that $\lambda $ is surjective. Consider the basis of the lattice $\mathscr{L}_1$ of $\K$ given by $$\overline{\mathcal{A}} = \left\{ 1, \psi_1, \psi_2 \right\} = \left\{ 1, \overline{a} \tau , \overline{a} \tau^2 + \overline{b} \tau \right\} .$$ The relationship between $\mathscr{L}$ with basis $\mathcal{A} = \left\{ 1, \phi_1, \phi_2 \right\}$ and $\mathscr{L}_1$ involves a change of lattice basis since 
\begin{equation}\label{basisch}
\Lambda_{\overline{\mathcal{A}}} = \Lambda_{\mathcal{A}} T , 
\end{equation}
where $T$ is given by \eqref{chbas}, and 
\begin{align*}
 \Lambda_{\mathcal{A}} & = \left(
\begin{array}{ccc}
 1 & \psi_1 & \psi_2 \\
 1 & \psi_1' & \psi_2' \\
 1 & \psi_1'' & \psi_2'' \\
\end{array}
\right) , &  \Lambda_{\mathcal{A}} & = \left(
\begin{array}{ccc}
 1 & \phi_1 & \phi_2 \\
 1 & \phi_1' & \phi_2' \\
 1 & \phi_1'' & \phi_2'' \\
\end{array}
\right) .
\end{align*}
Let $\alpha = u + x \phi_1 + y \phi_2 \in \mathcal{R}$. Multiplying \eqref{basisch} on the right by $T^{-1}$, we obtain an element $\gamma \in \mathcal{R}_{M}$ such that $\lambda (\gamma ) = \alpha $. Since we have shown that $\lambda $ is surjective, it follows that $\lambda$ is a ring isomorphism. 
\end{proof}

Theorem \ref{cubicpar} has the following converse. 
\begin{theorem}\label{convcub}
If $\mathcal{O}$ is an order of the cubic field $\K = \Q (\zeta )$, then there exists an integral binary cubic form $\overline{\mathcal{C}} = \left( \overline{a}, \overline{b}, \overline{c}, \overline{d} \right)$ such that $\overline{\mathcal{C}}$ parameterizes $\mathcal{O}$. 
\end{theorem}

\begin{proof}
Following Bhargava, Shankar, and Tsimerman \cite{Bhargone0}, we reiterate what was said about this. Let $\mathcal{O} = [1, \omega , \theta ]$. We can assume with no loss of generality that $\omega \theta \in \Z$ since if not, we can translate so that this occurs. Now we have 
\begin{eqnarray*}
     \omega^2 & = & w_{11} + w_{12} \omega + w_{13} \theta , \\
\omega \theta & = & w_{21} , \\
     \theta^2 & = & w_{31} + w_{32} \omega +  w_{33} \theta ,
\end{eqnarray*}
for some rational integers $w_{ij}$ ($1 \leq i, j \leq 3$). By setting $\overline{\mathcal{C}} = \left( \overline{a}, \overline{b}, \overline{c}, \overline{d} \right) = \left( w_{13}, - w_{12}, w_{33}, -w_{32} \right)$, we find that $w_{11} = - a c$, $w_{21} = - a d$, and $w_{31} = - b d$ so that $\overline{\mathcal{C}}$ parameterizes $\mathcal{O}$. More details on this calculation were given in \cite[12--13]{cfg}.  
\end{proof}

In order to see how to generalize this theorem with the arithmetic matrices, we will need to consider the result from a slightly different point of view. Again let $\mathcal{O} = [1, \omega , \theta ]$. The order $\mathcal{O}$ of $\K$ will have multiplication formulas which can be expressed in the following way.
\begin{eqnarray*}
u_1 + x_1 \omega + y_1 \theta & = & u_1 + x_1 \omega + y_1 \theta , \\
u_1 \omega + x_1 \omega^2 + y_1 \omega \theta & = & u_2 + x_2 \omega + y_2 \theta , \\ 
u_1 \theta + x_1 \omega \theta + y_1 \theta^2 & = & u_3 + x_3 \omega + y_3 \theta .
\end{eqnarray*}
Applying the embeddings of $\K$, we can extend this system of equations to write
\begin{eqnarray*}
& & \left(
\begin{array}{ccc}
 u_1 + x_1 \omega + y_1 \theta & 0 & 0 \\
 0 & u_1 + x_1 \omega' + y_1 \theta' & 0 \\
 0 & 0 & u_1 + x_1 \omega'' + y_1 \theta'' \\
\end{array}
\right) \left(
\begin{array}{ccc}
 1 & \omega & \theta \\
 1 & \omega' & \theta' \\
 1 & \omega'' & \theta'' \\
\end{array}
\right) \\
 & = & \left(
\begin{array}{ccc}
 1 & \omega & \theta \\
 1 & \omega' & \theta' \\
 1 & \omega'' & \theta'' \\
\end{array}
\right) \left(
\begin{array}{ccc}
 u_1 & u_2 & u_3 \\
 x_1 & x_2 & x_3 \\
 y_1 & y_2 & y_3 \\
\end{array}
\right) = \Gamma_{\mathcal{O}} U . 
\end{eqnarray*}
In solving for $u_2$, $x_2$, $y_2$, $u_3$, $x_3$, $y_3$, under the assumption that $\omega = \overline{a} \zeta$, $\theta = \overline{a} \zeta^2 + \overline{b} \zeta $ for some $\overline{a}, \overline{b} \in \Z$, we see that there must exist some rational integers $ \overline{c}, \overline{d}$ such that 
\begin{equation*}
\left(
\begin{array}{ccc}
 u_1 & u_2 & u_3 \\
 x_1 & x_2 & x_3 \\
 y_1 & y_2 & y_3 \\
\end{array}
\right) = N_{\mathcal{O}}^{ \left( u_1 + x_1 \omega + y_1 \theta \right) } = \left(
\begin{array}{ccc}
 u_1 & - \overline{a} \ \overline{d} y_1 & - \overline{d} \left( \overline{a} x_1 + \overline{b} y_1 \right) \\
 x_1 & u_1 - \overline{b} x_1 - \overline{c} y_1 & - \overline{c} x_1 - \overline{d} y_1 \\
 y_1 & \overline{a} x_1 & u_1 - \overline{c} y_1 \\
\end{array}
\right) .
\end{equation*}
Furthermore, $\left( \overline{a}, \overline{b}, \overline{c}, \overline{d} \right)$ will be irreducible in $\Q [x, y]$.

\section{Quartic rings}

It is natural to suspect that we can use exactly the same procedure with binary quartic forms to parametrize quartic rings. Let $\mathcal{V} = (a,b,c,d,e)$ be a binary quartic form so that $a, b, c, d, e$ are any four rational integers such that $$\mathcal{V}(x, y) = a x^4 + b x^3 y + c x^2 y^2 + d x y^3 + e x^4 $$ is irreducible over $\Q$. The formulas given in Proposition \ref{mainres} above provide the arithmetic matrix for the order $\mathcal{O}$ of the quartic field $\F = \Q (\zeta )$ given by
\begin{align}\label{arithquartic}
N_{\mathcal{O}}^{(\alpha )} & = \left(
\begin{array}{cccc}
 u & -a e z & -e (a y+b z) & -e (a x+b y+c z) \\
 x & u-b x-c y-d z & -c x-d y-e z & -d x-e y \\
 y & a x & u-c y-d z & -d y-e z \\
 z & a y & a x+b y & u-d z \\
\end{array}
\right) ,  
\end{align}
where $\alpha = u + x \phi_1 + y \phi_2 + z \phi_3 $, $\phi_1 = a \zeta $, $\phi_2 = a \zeta^2 + b \zeta $, $\phi_3 = a \zeta^3 + b \zeta^2 + c \zeta $. As earlier mentioned about cubic orders, when the discriminant of $\mathcal{V}$ is not equal to the discriminant of a quartic field, the integral basis $\mathcal{A} = \{ 1 , \phi_1 , \phi_2 , \phi_3 \}$ does not generate the entire ring of integers of the field $\F$, but just an order $\mathcal{O}$ of $\F$. To give multiplication formulas as we did for cubic rings, we compute the matrix products 
\small 
\begin{align*}
N_{\mathcal{O}}^{\left( \phi_1 \right) } \left(
\begin{array}{ccc}
 0 & 0 & 0 \\
 1 & 0 & 0 \\
 0 & 1 & 0 \\
 0 & 0 & 1 \\
\end{array}
\right) & = \left(
\begin{array}{ccc}
 0 & 0 & - a e \\
 -b & - c & - d \\
 a & 0 & 0 \\
 0 & a & 0 \\
\end{array}
\right) , N_{\mathcal{O}}^{\left( \phi_2 \right) } \left(
\begin{array}{cc}
 0 & 0 \\
 0 & 0 \\
 1 & 0 \\
 0 & 1 \\
\end{array}
\right) & = \left(
\begin{array}{cc}
 - a e & - b e \\
 - d & - e \\
 - c & - d \\
 b & 0 \\
\end{array}
\right) , \\
 N_{\mathcal{O}}^{\left( \phi_3 \right) } \left(
\begin{array}{c}
 0 \\
 0 \\
 0 \\
 1 \\
\end{array}
\right) & = \left(
\begin{array}{c}
 -c e \\
 0 \\
 -e \\
 -d \\
\end{array}
\right) . & & 
\end{align*}
\normalsize 
This gives the multiplication formulas for the ring generated by the integral basis $\mathcal{A}$ of $\mathcal{O}$. 
\small 
\begin{align*}
\phi_1^2 & = - b \phi_1 + a \phi_2 , & \phi_1 \phi_2 & = - c \phi_1 + a \phi_3 , & \phi_1 \phi_3 & = - a e - d \phi_1 , \\ 
\phi_2^2 & = -a e - d \phi_1 - c \phi_2 + b \phi_3 , & \phi_2 \phi_3 & =  - b e - e \phi_1 - d \phi_2 , & \phi_3^2 & = - c e - e \phi_2 - d \phi_3 .
\end{align*}
\normalsize 
Of course, it is much easier to use the matrices $N_{\mathcal{O}}^{(\alpha )}$ and $N_{\mathcal{O}}^{(\beta )}$ to take sums and products of $\alpha , \beta \in \mathcal{R}$ than to use the above formulas. 

Next we ask whether Theorem \ref{cubicpar} has a quartic analogue. Bhargava \cite{Bhargone3} writes about this question, "However, since the jump in complexity from $k = 3$ to $k = 4$ is so large, this idea goes astray very quickly (yielding a huge mess!), and it becomes necessary to have a new perspective in order to make any further progress." The arithmetic matrices help quite a lot to condense the mess before it happens, however the formulas are still rather large when multiplied out. In order to understand how to find the appropriate change of basis, we recall how the matrix $T$ in \eqref{chbas} was found. Since 
\begin{align*}
 \tau & = \frac{s \zeta - q}{- r \zeta + p} , & \phi_1 & = a \zeta , & \phi_2 & = a \zeta^2 + b \zeta , 
\end{align*}
and $(a,b,c,d) \circ M = \left( a_1 , b_1 , c_1 , d_1 \right)$, where 
\begin{eqnarray*}
a_1 & = & a p^3 + b p^2 r + c p r^2 + d r^3 , \\
b_1 & = & 3 a p^2 q + b p^2 s + 2 b p q r + 2 c p r s + c q r^2 + 3 d r^2 s , 
\end{eqnarray*}
we are able to solve the following system for rational integers $t_{12}, t_{22}, t_{32}, t_{13}, t_{23}, t_{33}$. 
\begin{eqnarray*}
a_1 \tau & = & t_{12} + t_{22} \phi_1 + t_{32} \phi_2 , \\
a_1 \tau^2 + b_1 \tau & = & t_{13} + t_{23} \phi_1 + t_{33} \phi_2 .
\end{eqnarray*}
Using the arithmetic matrices to do so greatly simplifies the calculation. 

Returning now to the question of a quartic analogue of Theorem \ref{cubicpar}, let $M = \left(
\begin{array}{cc}
 p & q \\
 r & s \\
\end{array}
\right)$ and let $\mathcal{V} = (a,b,c,d,e)$ be irreducible in $\Q [x, y]$ with $a,b,c,d,e \in \Z$, $a e \not= 0$. Let $\mathcal{V}_1 = \mathcal{V} \circ M = \left( a_1, b_1, c_1, d_1, e_1 \right)$, where the substitution $x \longmapsto p x + q y$, $y \longmapsto r x + s y$ gives 
\small
\begin{eqnarray*}
\overline{a} & = & a p^4 + b p^3 r + c p^2 r^2 + d p r^3 + e r^4 , \\
\overline{b} & = & 4 a p^3 q + b p^3 s + 3 b p^2 q r + 2 c p^2 r s + 2 c p q r^2 + 3 d p r^2 s + d q r^3 + 4 e r^3 s , \\
\overline{c} & = & 6 a p^2 q^2+3 b p^2 q s+3 b p q^2 r+c p^2 s^2+4 c p q r s+c q^2 r^2+3 d p r s^2+3 d q r^2 s+6 e r^2 s^2 , \\
\overline{d} & = & 4 a p q^3+3 b p q^2 s+b q^3 r+2 c p q s^2+2 c q^2 r s+d p s^3+3 d q r s^2+4 e r s^3 , \\
\overline{e} & = & a q^4 +b q^3 s+c q^2 s^2+d q s^3+e s^4 . 
\end{eqnarray*}
\normalsize 
Let 
\begin{align*}
 \tau & = \frac{s \zeta - q}{- r \zeta + p} , & \phi_1 & = a \zeta , & \phi_2 & = a \zeta^2 + b \zeta , & \phi_3 & = a \zeta^3 + b \zeta^2 + c \zeta . 
\end{align*}
Then we must solve the system 
\begin{eqnarray}
\label{systone} \overline{a} \tau & = & t_{12} + t_{22} \phi_1 +  t_{32} \phi_2 +  t_{42} \phi_3 = \psi_2 , \\
\label{systtwo} \overline{a} \tau^2 + \overline{b} \tau & = & t_{13} + t_{23} \phi_1 +  t_{33} \phi_2 +  t_{43} \phi_3  = \psi_3, \\
\label{systthr} \overline{a} \tau^3 + \overline{b} \tau^2 + \overline{c} \tau & = & t_{14} + t_{24} \phi_1 +  t_{34} \phi_2 +  t_{44} \phi_3  = \psi_4 ,
\end{eqnarray} 
for the $t_{ij}$ so that the matrix $T = \left[ t_{ij} \right]$ with $t_{11} = 1$ and $t_{i1} = 0$, belongs to $\text{GL}_4(\Z )$. The system of equations \eqref{systone}, \eqref{systtwo}, \eqref{systthr} can be expressed as
\begin{equation}\label{solvB}
\Xi_{\mathcal{O}} Q B = \Theta_{\mathcal{O}}^{ \left( (p - r \zeta )^3 \right) } \Xi_{\mathcal{O}} A T  ,
\end{equation}
where $\zeta_0 = \zeta $, $\zeta_1, \zeta_2, \zeta_3 $ are the roots of $\mathcal{V}(x, 1)$, and 
\small 
\begin{align*}
\Xi_{\mathcal{O}} & = \left(
\begin{array}{cccc}
 1 & \zeta_0 & \zeta_0^2 & \zeta_0^3 \\
 1 & \zeta_1 & \zeta_1^2 & \zeta_1^3 \\
 1 & \zeta_2 & \zeta_2^2 & \zeta_2^3 \\
 1 & \zeta_3 & \zeta_3^2 & \zeta_3^3 \\
\end{array}
\right) ,  & A & = \left(
\begin{array}{cccc}
 1 & 0 & 0 & 0 \\
 0 & a & b & c \\
 0 & 0 & a & b \\
 0 & 0 & 0 & a \\
\end{array}
\right) ,  & B & = \left(
\begin{array}{cccc}
 1 & 0 & 0 & 0 \\
 0 & \overline{a} & \overline{b} & \overline{c} \\
 0 & 0 & \overline{a} & \overline{b} \\
 0 & 0 & 0 & \overline{a} \\
\end{array}
\right) , 
\end{align*}
\normalsize 
\begin{eqnarray*}
\Theta_{\mathcal{O}}^{ \left( (p - r \zeta )^3 \right) } & = & \left(
\begin{array}{cccc}
 \left( p - r \zeta_0 \right)^3 & 0 & 0 & 0 \\
 0 & \left( p - r \zeta_1 \right)^3 & 0 & 0 \\
 0 & 0 & \left( p - r \zeta_2 \right)^3 & 0 \\
 0 & 0 & 0 & \left( p - r \zeta_3 \right)^3 \\
\end{array}
\right) ,  \\
Q & = &
\left(
\begin{array}{cccc}
 p^3 & -p^2 q & p q^2 & -q^3 \\
 -3 p^2 r & s p^2+2 q r p & -2 p s q-q^2 r & 3 q^2 s \\
 3 p r^2 & -q r^2-2 p s r & p s^2+2 q r s & -3 q s^2 \\
 -r^3 & r^2 s & -r s^2 & s^3 \\
\end{array}
\right) .
\end{eqnarray*}
Taking determinants in \eqref{solvB} shows that $\det (T) = m^6 = 1$. So any solution $T$ to \eqref{solvB} must belong to $\text{SL}_4(\Z )$. Solving \eqref{solvB} for $T$, 
\begin{eqnarray}
\label{Bsolved} T & = & \left( \Xi_{\mathcal{O}} A \right)^{-1} \Theta_{\mathcal{O}}^{ \left( (p - r \zeta )^{-3} \right) } \left( \Xi_{\mathcal{O}} A \right) \left( A^{-1} Q B \right) , \\ 
\nonumber                  & = & \Gamma_{\mathcal{O}}^{-1} \Theta_{\mathcal{O}}^{ \left( (p - r \zeta )^{-3} \right) } \Gamma_{\mathcal{O}} \left( A^{-1} Q B \right) , \\ 
\nonumber                  & = & \left( N_{\mathcal{O}}^{ \left( p - r \zeta \right) } \right)^{-3} \left( A^{-1} Q B \right) .
\end{eqnarray}
Now the matrix $\overline{a} \left( N_{\mathcal{O}}^{ \left( p - r \zeta \right) } \right)^{-1}$ is equal to
\small
\begin{equation*}
 \left(
\begin{array}{cccc}
 a p^3+b p^2 r+c p r^2+d r^3 & -a e r^3 & -e r^2 (a p+b r) & -e r \left( a p^2+b p r+c r^2 \right) \\
 p^2 r & a p^3 & -r \left( c p^2+d p r+e r^2\right) & -p r (d p+e r) \\
 p r^2 & a p^2 r & p^2 (a p+b r) & -r^2 (d p+e r) \\
 r^3 & a p r^2 & p r (a p+b r) & p \left( a p^2+b p r+c r^2\right) \\
\end{array}
\right) .
\end{equation*}
\normalsize 
Expanding $\left( N_{\mathcal{O}}^{ \left( p - r \zeta \right) } \right)^{-3} \left( A^{-1} Q B \right)$ shows that 
\begin{equation}\label{quartbasch}
T = \left(
\begin{array}{cccc}
 1 & t_{12} & t_{13} & t_{14} \\
 0 & m p^2 & 2 m p q & m q^2 \\
 0 & m p r & m (p s+q r) & m q s \\
 0 & m r^2 & 2 m r s & m s^2 \\
\end{array}
\right) ,
\end{equation}
where 
\begin{eqnarray*}
m & = & \det (M) = p s - q r , \\
t_{12} & = & -\left( a p^3 q + b p^2 q r + c p q r^2 + d q r^3 + e r^3 s \right) , \\
t_{13} & = & -\left( 3 a p^2 q^2 + b p^2 q s + 2 b p q^2 r + 2 c p q r s + c q^2 r^2 + 3 e r^2 s^2 + 3 d q r^2 s \right) , \\
t_{14} & = & -\left( 3 a p q^3 + b q^3 r + 2 b p q^2 s + 2 c q^2 r s + c p s^2 q + 3 d q r s^2 + 3 e r s^3 \right) .
\end{eqnarray*}
Thus, using $T$ as a change of basis matrix, it is possible to prove the following generalization of Theorem \ref{cubicpar}. 

\begin{theorem}\label{quarticpar}
Let $\mathcal{V} = (a,b,c,d,e)$ be a binary quartic form so that $a,b,c,d,e$ are any rational integers such that $a e \not= 0$ and $\mathcal{V}(x, y)$ is irreducible in $\Q [x, y]$. Let $$M = \left(
\begin{array}{cc}
 p & q \\
 r & s \\
\end{array}
\right) \in \text{GL}_2(\Z ) .$$ Then the binary quartic form $\mathcal{V} \circ M$ parameterizes a quartic ring $\mathcal{R}_{M} = \left[ 1 , \psi_1, \psi_2, \psi_3 \right] $ that is isomorphic to the quartic ring $\mathcal{R} = \left[ 1 , \phi_1, \phi_2, \phi_3 \right] $ parameterized by $\mathcal{V}$. 
\end{theorem}

An analogue of Theorem \ref{convcub} also holds for quartic rings since the argument after the proof easily extends to rings of rank an arbitrary positive integer $n \geq 3$, provided that the $\phi_j$ are known to be of the form \eqref{genform}. Bhargava \cite{Bhargone3} found that certain pairs of ternary quadratic forms parameterize quartic rings, those pairs that are linearly independent over $\Q$. However, this may still be made to agree, as we might show this using syzygys of classical invariant theory. Recall from \cite{hamblarithnf} that the Diophantine equation $ \det \left( N_{\mathcal{O}}^{(\alpha )} \right) = 1$ can be expressed, where $t$ is the trace of $\alpha$, as
\begin{equation}\label{quartnorm}
t^4 - 2 \mathcal{G} t^2 - 8 \mathcal{H} t  + \mathcal{F} = 256 .
\end{equation}
This equation is analogous to the Pell equation $t^2 - \Delta y^2 = 4$, where $t$ is the trace of $\frac{1}{2} \left( x + y \sqrt{\Delta } \right)$. In the case of quartics,
\begin{eqnarray*}
\mathcal{G}(x, y, z) & = & (3 b^2 - 8 a c) x^2 + (4 b c - 24 a d) x y + (4 c^2 - 8 b d - 16 a e) y^2 \\
                     &   & + (2 b d - 32 a e) x z + (4 c d - 24 b e) y z + (3 d^2 - 8 c e) z^2 ,
\end{eqnarray*}
$\mathcal{H}(x, y, z)$ is a homogeneous ternary cubic polynomial, and $\mathcal{F}(x, y, z)$ is a homogeneous ternary quartic polynomial. Let $I$ and $J$ denote the invariants
\begin{align*}
I & = 12 a e - 3 b d + c^2, & J & = 72 a c e + 9 b c d - 27 a d^2 - 27 b^2 e - 2 c^3 .
\end{align*}
Following Cremona \cite{Cremona}, the ternary forms $\mathcal{F}(x, y, z)$, $\mathcal{G}(x, y, z)$, $\mathcal{H}(x, y, z)$ satisfy the syzygy 
\begin{equation*}
g_4^3 - 48 g_4 I v^2 - 64 J v^3 = 27 g_6^2 
\end{equation*}
where $g_4 = \mathcal{G} \left( x^2, x, 1 \right)$, $g_6 = \mathcal{H} \left( x^2, x, 1 \right)$, and $v = \mathcal{V} (x, 1)$. Thus is seems plausible that there may be a relationship between the binary quartic form $\mathcal{V} (x, y)$ and a pair of ternary quadratic forms. However, the rings parameterized by $\mathcal{V} (x, y)$ are those of the form $\left[ 1, a \zeta , a \zeta^2 + b \zeta , a \zeta^3 + b \zeta^2 + c \zeta \right]$. We have not shown that all quartic rings can be expressed in this way.

\section{Quintic rings}

Parametrization of quintic rings has also been investigated by Bhargava \cite{Bhargone4}. In this section we compute the change of basis matrix for the quintic ring parameterized by the binary quintic form $\mathcal{Q}(x, y) = (a,b,c,d,e,f)$. Following the same recipe as in earlier sections, let
\begin{equation*}
Q = \left(
\begin{array}{ccccc}
 p^4 & -p^3 q & p^2 q^2 & -p q^3 & q^4 \\
 -4 p^3 r & p^3 s + 3 p^2 q r & - 2 p^2 q s - 2 p q^2 r  & q^3 r + 3 p q^2 s & - 4 q^3 s \\
 6 p^2 r^2 & - 3 p^2 r s - 3 p q r^2 & q^2 r^2 + 4 p q r s + p^2 s^2 & - 3 q^2 r s - 3 p s^2 q & 6 q^2 s^2 \\
 -4 p r^3 & q r^3 + 3 p r^2 s & -2 q r^2 s - 2 p r s^2 & p s^3 + 3 q r s^2 & - 4 q s^3 \\
 r^4 & - r^3 s & r^2 s^2 & - r s^3 & s^4 \\
\end{array}
\right) .
\end{equation*}
The element in the $i$-th row and $j$-th column of $Q$ is the coefficient of $\zeta^{i-1}$ in the expansion of $(-r \zeta + p)^{n - j} (s \zeta - q)^{j - 1}$. 

We must calculate $T = \left( N_{\mathcal{O}}^{ \left( p - r \zeta \right) } \right)^{-4} \left( A^{-1} Q B \right)$. Doing so gives
\begin{eqnarray*}
t_{12} & = & -a q p^4-b q r p^3-c q r^2 p^2-d q r^3 p-e q r^4-f r^4 s , \\
t_{13} & = & -4 a q^2 p^3-b q s p^3-3 b q^2 r p^2-2 c q r s p^2-2 c q^2 r^2 p-3 d q r^2 s p-d q^2 r^3 \\
       &   & -4 f r^3 s^2-4 e q r^3 s , \\
t_{14} & = & -6 a p^2 q^3-c r^2 q^3-3 b p r q^3-3 b p^2 s q^2-3 d r^2 s q^2-4 c p r s q^2-c p^2 s^2 q \\
       &   & -6 e r^2 s^2 q-3 d p r s^2 q-6 f r^2 s^3 , \\
t_{15} & = & -4 a p q^4-b r q^4-3 b p s q^3-2 c r s q^3-2 c p s^2 q^2-3 d r s^2 q^2-d p s^3 q \\
       &   & -4 e r s^3 q-4 f r s^4 ;
\end{eqnarray*}
\begin{equation*}
T = \left(
\begin{array}{ccccc}
 1 & t_{12} & t_{13} & t_{14} & t_{15} \\
 0 & p^3 m & 3 p^2 q m & 3 p q^2 m & q^3 m \\
 0 & p^2 r m & p m (p s + 2 q r) & q m (2 p s + q r) & q^2 s m \\
 0 & p r^2 m & r m (2 p s + q r) & s m (p s + 2 q r) & q s^2 m \\
 0 & r^3 m & 3 r^2 s m & 3 r s^2 m & s^3 m \\
\end{array}
\right) . 
\end{equation*}
$T$ is the change of basis for parametrization of quintic rings by binary quintic forms, since when $m = \det (M)$, $M \in \text{GL}_2(\Z )$, $\det (T) = m^{10} = 1$ so $T \in \text{SL}_{5}(\Z )$.  

\section{Rings of rank $n$}

Now that we understand how to obtain a change of basis matrix $T$ under the assumption that there is a ring bijection, we are able to see the general formula for the change of basis matrix $T$ for an arbitrary $n > 2$. In this section we will start by defining the matrix $T$ in a way that we expect it to be a change of basis matrix, and show that a calculation proves that it is. We begin this section with the following remark. 

\begin{remark}
Let $\F$ be a number field of degree $n$ over $\mathbb{Q}$, and let the collection of algebraic integers $\mathcal{A} = \{ 1, \phi_1, \phi_2, \dots , \phi_{n-1} \}$ be an integral basis of an order $\mathcal{O}$ of $\F$. Let $T \in \text{Sl}_{n}(\mathbb{Z})$. If the lattice generated by $ (1, \phi_1, \phi_2, \dots , \phi_{n-1}) T$ is also an order $\mathcal{O}'$ of $\F$, then there is ring isomorphism $\lambda_{T} : \mathcal{O} \longrightarrow \mathcal{O}'$ given by $\lambda_{T} : \mathcal{A} \longmapsto \overline{\mathcal{A}}$, where $\overline{\mathcal{A}}$ is the basis obtained by expanding $(1, \phi_1, \phi_2, \dots , \phi_{n-1}) T$. Injectivity and surjectivity of $\lambda_{T}$ are clear since $T \in \text{Sl}_{n}(\mathbb{Z})$. We obtain the ring isomorphism by transport of the ring structure of $\mathcal{O}$ and verifying that the imposed structure on $\mathcal{O}'$ agrees with the original structure of $\mathcal{O}'$.  
\end{remark} 

Define a map 
\begin{align*}
 \omega & : \text{GL}_2 (\Z ) \longrightarrow \text{SL}_n(\Z ), & \omega & : M \longmapsto T , 
\end{align*}
where 
\begin{align*}
M & = \left(
\begin{array}{cc}
 p & q \\
 r & s \\
\end{array}
\right) , & T & = \left(
\begin{array}{cc}
 1 & G \\
 H & m P \\
\end{array}
\right) , 
\end{align*}
$H = (0,0, \dots , 0)^T$, $m = \det (M)$, $G = \left( t_{12}, t_{13}, \dots , t_{1n} \right)$, and $P$ is an $n-1$ by $n-1$ matrix also of determinant $m$ for which we give a general description below. Observe that $\omega $ is well defined since $\det (T) = m \det (P) = 1$ and after properly defining $P$, it will be clear that $\omega $ is injective.  

When $n = 3$, and $4$ respectively we have 
\begin{align*}
P & = \left(
\begin{array}{cc}
  p & q  \\
  r & s  \\
\end{array}
\right) = M , & P & = \left(
\begin{array}{ccc}
  p^2 & 2 p q & q^2 \\
  p r & p s + q r & q s \\
  r^2 & 2 r s & s^2 \\
\end{array}
\right) .
\end{align*}
When $n = 5$, we have 
\begin{equation*}
P = \left(
\begin{array}{cccc}
  p^3 & 3 p^2 q & 3 p q^2 & q^3 \\
  p^2 r & 2 p q r + p^2 s & q^2 r + 2 p q s & q^2 s \\
  p r^2 & q r^2 + 2 p r s & 2 q r s + p s^2 & q s^2 \\
  r^3 & 3 r^2 s & 3 r s^2 & s^3 \\
\end{array}
\right) .
\end{equation*}
Entries in the $i$-th row and $j$-th column of the $n - 1$ by $n - 1$ matrix $P$ can be seen as the coefficient of $x^{j-1}$ in the expansion of $(p + q x)^{n - 1 - i} (r + s x)^{i - 1} $, equal to
\begin{equation*}
\left( \sum_{h = 1}^{n + 1 - i} \binom{n - 1 - i}{h - 1} p^{n - i - h} q^{h - 1} x^{h - 1} \right) \left( \sum_{k = 1}^{i} \binom{i-1}{k-1} r^{i-k} s^{k-1} x^{k-1} \right) . 
\end{equation*}
The entries in the first row of $T$, noting that the following $j$ is not the same as the $j$ in the description of $P$ as they are different by $1$, are the $t_{1j}$ for $j = 2, 3, \dots , n+1$, defined by 
\begin{equation*}
\sum_{j=2}^{n+1} t_{1j} x^{j-2} = - q \mathcal{B}_{n-1} (p + q x, r + s x) - a_{n+1} s (r + s x)^3 ,
\end{equation*}
where we can discard $t_{1,n+1}$. This completes the general definition of the change of basis matrix $T$. We define the $n \times n$ matrix $Q$ by stating that the element in the $i$-th row and $j$-th column of $Q$ is the coefficient of $x^{i-1}$ in the expansion of $(-r x + p)^{n - j} (s x - q)^{j - 1}$. We let
\begin{align*}
A & = \left(
\begin{array}{cccccc}
 1 & 0 & 0 & 0 & \dots & 0 \\
 0 & a_1 & a_2 & a_3 & \dots & a_{n-1} \\
 0 & 0 & a_1 & a_2 & \dots & a_{n-2} \\
 0 & 0 & 0 & a_1 & \dots & a_{n-3} \\
 \vdots & \vdots & \vdots & \vdots & \ddots & \vdots \\
 0 & 0 & 0 & 0 & \dots & a_1 \\
\end{array}
\right) , & B & = \left(
\begin{array}{cccccc}
 1 & 0 & 0 & 0 & \dots & 0 \\
 0 & b_1 & b_2 & b_3 & \dots & b_{n-1} \\
 0 & 0 & b_1 & b_2 & \dots & b_{n-2} \\
 0 & 0 & 0 & b_1 & \dots & b_{n-3} \\
 \vdots & \vdots & \vdots & \vdots & \ddots & \vdots \\
 0 & 0 & 0 & 0 & \dots & b_1 \\
\end{array}
\right) , 
\end{align*}
where the $b_j$ ($j = 1, \dots , n - 1$) are given by $\mathcal{B}_n \circ M = \left( b_1, b_2, \dots , b_{n+1} \right)$, and 
\begin{equation*}
p - r \zeta = \frac{1}{a_1} \left( a_1 p - r \phi_1 \right) .
\end{equation*} 
The coefficients $b_1, b_2, \dots b_{n+1}$ can be obtained by expanding the right hand side of 
\begin{equation*}
\sum_{i = 1}^{n+1} b_i x^{i-1} = \mathcal{B}_n (p + q x, r + s x) .
\end{equation*}
We have
\begin{equation*}
N_{\mathcal{O}}^{ \left( a_1 p - r \phi_1 \right) } = \left(
\begin{array}{cccccccc}
 p a_1 & 0 & 0 & 0 & \dots & 0 & 0 & r a_1 a_{n+1} \\
 -r & p a_1+r a_2 & r a_3 & r a_4 & \dots & r a_{n-2} & r a_{n-1} & r a_n \\
 0 & -r a_1 & p a_1 & 0 & \dots & 0 & 0 & 0 \\
 0 & 0 & -r a_1 & p a_1 & \dots & 0 & 0 & 0 \\
 \vdots & \vdots & \vdots & \vdots & \ddots & \vdots & \vdots & \vdots \\
 0 & 0 & 0 & 0 & \dots & -r a_1 & p a_1 & 0 \\
 0 & 0 & 0 & 0 & \dots & 0 & -r a_1 & p a_1 \\
\end{array}
\right) .
\end{equation*}
In order to complete the proof of Proposition \ref{mainparam}, we must show that \eqref{siam} below holds. This means that the proof of Proposition \ref{mainparam} is reduced to a symbolic computation with matrices. The calculation in \eqref{Bsolved} will work more generally, as the following lemma shows.

\begin{lemma}\label{toprop}
Let $n$ be a positive integer greater than $1$. Let $A$, $B$, $M$, $Q$, $T = \left[ t_{ij} \right]$ and $\mathcal{B}(x, y) = \left( a_1, a_2, \dots , a_{n+1} \right)$ be as defined in this section. If 
\begin{equation}\label{siam}
a_1^{n-1} A^{-1} Q B T^{-1} = \left( N_{\mathcal{O}}^{\left( a_1 p - r \phi_1 \right) } \right)^{n-1} , 
\end{equation}
then Proposition \ref{mainparam} holds for that $n$ (which can exceed $6$). 
\end{lemma}

\begin{proof}
First observe that if \eqref{siam} holds, then $T \in \text{SL}_{n}(\Z )$ by taking determinants. The calculation in \eqref{Bsolved} will work in general. Consider the system of equations 
\small
\begin{eqnarray*}
 b_1 \tau & = & t_{1,2} + t_{2,2} \phi_1 +  t_{3,2} \phi_2 + \dots + t_{n,2} \phi_{n-1} = \psi_2 , \\
 b_1 \tau^2 + b_2 \tau & = & t_{1,3} + t_{2,3} \phi_1 +  t_{3,3} \phi_2 + \dots + t_{n,3} \phi_{n-1} = \psi_3, \\
 b_1 \tau^3 + b_2 \tau^2 + b_3 \tau & = & t_{1,4} + t_{2,4} \phi_1 +  t_{3,4} \phi_2 + \dots + t_{n,4} \phi_{n-1}  = \psi_4 , \\
 \vdots & & \vdots \\
 b_1 \tau^{n-1} + b_2 \tau^{n-2} + \dots + b_{n-1} \tau & = & t_{1,n} + t_{2,n} \phi_1 +  t_{3,n} \phi_2 + \dots + t_{n,n} \phi_{n-1}  = \psi_{n-1} ,
\end{eqnarray*} 
\normalsize 
where $\tau = \frac{s \zeta - q}{- r \zeta + p}$ is a root of $\mathcal{B}_n(x, y) \circ M$. We can express this system of equations in matrix form as 
\begin{equation}\label{putmatr}
\Xi_{\mathcal{O}} Q B = \Theta_{\mathcal{O}}^{ \left( (p - r \zeta )^{n-1} \right) } \Xi_{\mathcal{O}} A T  ,
\end{equation}
where $\zeta_0 = \zeta $, $\zeta_1, \zeta_2, \dots , \zeta_{n-1} $ are the roots of $\mathcal{B}_n(x, 1)$, $\tau_0 \ ( = \tau )$, $\tau_1, \dots , \tau_{n-1}$ are given by $\tau_j = \frac{s \zeta_j - q}{-r \zeta_j + p}$, and  
\begin{eqnarray*}
\Xi_{\mathcal{O}} & = & \left(
\begin{array}{ccccc}
 1 & \zeta_0 & \zeta_0^2 & \dots & \zeta_0^{n-1} \\
 1 & \zeta_1 & \zeta_1^2 & \dots & \zeta_1^{n-1} \\
 1 & \zeta_2 & \zeta_2^2 & \dots & \zeta_2^{n-1} \\
 \vdots & \vdots & \vdots & \ddots & \vdots \\
 1 & \zeta_{n-1} & \zeta_{n-1}^2 & \dots & \zeta_{n-1}^{n-1} \\
\end{array}
\right) , \\
    &  & \\
\Theta_{\mathcal{O}}^{ \left( (p - r \zeta )^{1-n} \right) } \Xi_{\mathcal{O}} Q & = & \left(
\begin{array}{ccccc}
 1 & \tau_0 & \tau_0^2 & \dots & \tau_0^{n-1} \\
 1 & \tau_1 & \tau_1^2 & \dots & \tau_1^{n-1} \\
 1 & \tau_2 & \tau_2^2 & \dots & \tau_2^{n-1} \\
 \vdots & \vdots & \vdots & \ddots & \vdots \\
 1 & \tau_{n-1} & \tau_{n-1}^2 & \dots & \tau_{n-1}^{n-1} \\
\end{array}
\right) . 
\end{eqnarray*}
As we calculated in the special case of $n = 4$ in \eqref{Bsolved}, we find that the system of equations for $\psi_2, \psi_3, \dots , \psi_{n-1}$ is equivalent to
\begin{eqnarray*}
 T = \left( N_{\mathcal{O}}^{ \left( p - r \zeta \right) } \right)^{1-n} \left( A^{-1} Q B \right) .
\end{eqnarray*}
It follows that if $A,B,Q,M,T$ are defined as in this section and they satisfy \eqref{siam}, then $T \in \text{SL}_n(\Z )$ so that multiplication by $T$ gives a ring bijection between $\mathcal{O}$ generated by $\left\{ 1, \phi_1, \phi_2, \dots , \phi_{n-1} \right\}$ and $\mathcal{O}'$ generated by $\left\{ 1, \psi_1, \psi_2, \dots , \psi_{n-1} \right\}$. This bijection preserves the additive and multiplicative groups of each ring so we have a ring isomorphism. 
\end{proof}

\section{Proof of the main proposition}

In order to understand how to verify \eqref{siam}, and hence prove Proposition \ref{mainparam}, we will consider the calculation that verifies it for $n = 3$, $4$, $5$, and $6$. When $n = 3$, 
\begin{align*}
A & = \left(
\begin{array}{ccc}
 1 & 0 & 0 \\
 0 & a_1 & a_2 \\
 0 & 0 & a_1 \\
\end{array}
\right) , & N_{\mathcal{O}}^{\left( a_1 p - r \phi_1 \right) } & = \left(
\begin{array}{ccc}
 p a_1 & 0 & r a_1 a_4 \\
 -r & p a_1+r a_2 & r a_3 \\
 0 & -r a_1 & p a_1 \\
\end{array}
\right) , \\
 T & = \left(
\begin{array}{ccc}
 1 & t_{12} & t_{13} \\
 0 & m p  & m q  \\
 0 & m r  & m s  \\
\end{array}
\right) , & Q & = \left(
\begin{array}{ccc}
 p^2 & -p q & q^2 \\
 -2 p r & p s + q r & -2 q s \\
 r^2 & -r s & s^2 \\
\end{array}
\right) , 
\end{align*}
where 
\begin{eqnarray*}
t_{12} & = & -\left( a_1 p^2 q+a_2 p q r+a_3 q r^2+a_4 r^2 s \right) , \\
t_{13} & = & -\left( 2 a_1 p q^2+a_2 p q s+a_2 q^2 r+2 a_3 q r s+2 a_4 r s^2 \right) .
\end{eqnarray*}
$\left( N_{\mathcal{O}}^{\left( a_1 p - r \phi_1 \right) } \right)^2$ is equal to
\small
\begin{equation*}
 \left(
\begin{array}{ccc}
 p^2 a_1^2 & -r^2 a_1^2 a_4 & 2 p r a_1^2 a_4 \\
 -r \left( 2 p a_1+r a_2 \right) & a_2^2 r^2-a_1 a_3 r^2+2 p a_1 a_2 r+p^2 a_1^2 & r \left( 2 p a_1 a_3+r a_2 a_3-r a_1 a_4\right) \\
 r^2 a_1 & -r a_1 \left( 2 p a_1+r a_2 \right) & a_1 \left( p^2 a_1-r^2 a_3 \right) \\
\end{array}
\right) , 
\end{equation*}
\normalsize 
and the matrix $B$ is equal to
\footnotesize
\begin{equation*}
\left(
\begin{array}{ccc}
 1 & 0 & 0 \\
 0 & a_1 p^3 + a_2 p^2 r + a_3 p r^2 + a_4 r^3 & 3 a_1 p^2 q + a_2 p^2 s + 2 a_2 p q r + 2 a_3 p r s + a_3 q r^2 + 3 a_4 r^2 s \\
 0 & 0 & a_1 p^3 + a_2 p^2 r + a_3 p r^2 + a_4 r^3 \\
\end{array}
\right) .
\end{equation*}
\normalsize 
Let 
\begin{eqnarray*}
u_{12} & = & - \left( - a_1 p^2 q s + 2 a_1 p q^2 r + a_2 q^2 r^2 + a_3 q r^2 s + a_4 r^2 s^2 \right) , \\
u_{13} & = & - \left( - a_2 p^2 q s - 2 a_3 p q r s - 2 a_4 p r s^2 - a_1 p^2 q^2 + a_3 q^2 r^2 + a_4 q r^2 s \right) .
\end{eqnarray*}
We calculate
\begin{align*}
a_1^{3-1} A^{-1} & = \left(
\begin{array}{ccc}
 a_1^2 & 0 & 0 \\
 0 & a_1 & -a_2 \\
 0 & 0 & a_1 \\
\end{array}
\right) , & T^{-1} & =  \left(
\begin{array}{ccc}
 1 & u_{12} & u_{13} \\
 0 & s & -q \\
 0 & -r & p \\
\end{array}
\right) , 
\end{align*}
\begin{equation*}
Q B T^{-1} = \left(
\begin{array}{ccc}
 p^2 & - a_4 r^2 & 2 a_4 p r \\
 -2 p r & a_1 p^2 - a_3 r^2 & a_2 p^2 + 2 a_3 p r - a_4 r^2 \\
 r^2 & - r \left( 2 a_1 p + a_2 r \right) & a_1 p^2 - a_3 r^2  \\
\end{array}
\right) .
\end{equation*}
Expanding the product of these two matrices, we obtain \eqref{siam} for $n = 3$. 

Now let $n = 4$ and $P = \left[ p_{ij} \right]$. Then the entries of the $n - 1$ by $n - 1$ matrix $P$ are defined as coefficients of the polynomials (for $i = 1$ to $n - 1$)
\begin{equation*}
\sum_{j=1}^{4-1} p_{ij} x^{j-1} = (p + q x)^{n-1-i} (r + s x)^{i - 1} . 
\end{equation*}
We define
\begin{equation*}
T = \left(
\begin{array}{cc}
 1 & G \\
 H & m P \\
\end{array}
\right) = \left(
\begin{array}{cccc}
 1 & t_{12} & t_{13} & t_{14} \\
 0 & m p^2 & 2 m p q & m q^2 \\
 0 & m p r & m (p s + q r) & m q s \\
 0 & m r^2 & 2 m r s & m s^2 \\
\end{array}
\right) ,
\end{equation*}
where $t_{1j}$ for $j = 2, 3, \dots , 5$ are defined by 
\begin{equation*}
\sum_{j=2}^{n+1} t_{1j} x^{j-2} = - q \mathcal{B}_{3} (p + q x, r + s x) - a_{5} s (r + s x)^3 ,
\end{equation*}
discarding $t_{1,5}$. 
\begin{equation*}
a_1^3 A^{-1} = \left(
\begin{array}{cccc}
 a_1^3 & 0 & 0 & 0 \\
 0 & a_1^2 & -a_1 a_2 & a_2^2-a_1 a_3 \\
 0 & 0 & a_1^2 & -a_1 a_2 \\
 0 & 0 & 0 & a_1^2 \\
\end{array}
\right) .
\end{equation*}
Letting $Q = \left[ q_{ij} \right]$, the entries of $Q$ are defined as coefficients of the polynomials (for $j = 1$ to $4$)
\begin{equation*}
\sum_{i=1}^n q_{ij} x^{i-1} = (-r x + p)^{n-j} (s x - q)^{j - 1} . 
\end{equation*}
When $n = 4$, 
\begin{equation*}
Q = \left(
\begin{array}{cccc}
 p^3 & - p^2 q & p q^2 & -q^3 \\
 -3 p^2 r & p^2 s + 2 p q r & - 2 p q s - q^2 r & 3 q^2 s \\
 3 p r^2 & -q r^2 - 2 p s r & p s^2 + 2 q r s & - 3 q s^2 \\
 -r^3 & r^2 s & - r s^2 & s^3 \\
\end{array}
\right) . 
\end{equation*}
The matrix $B$ has entries which include the $b_i$ for $i = 1$ to $5$, defined by the coefficients of the polynomial 
\begin{align*}
\sum_{i = 1}^{n+1} b_i x^{i-1} & = \mathcal{B}_n (p + q x, r + s x) ; & B & = \left(
\begin{array}{cccc}
 1 & 0 & 0 & 0 \\
 0 & b_1 & b_2 & b_3 \\
 0 & 0 & b_1 & b_2 \\
 0 & 0 & 0 & b_1 \\
\end{array}
\right) . 
\end{align*}
In this case we expand $Q B T^{-1} $, equal to
\footnotesize 
\begin{equation*}
\left(
\begin{array}{cccc}
 p^3 & a_5 r^3 & - 3 a_5 p r^2 & 3 a_5 p^2 r \\
 -3 p^2 r & a_1 p^3 + a_4 r^3 & a_2 p^3 - 3 a_4 p r^2 + a_5 r^3 & p \left( a_3 p^2 + 3 a_4 p r - 3 a_5 r^2 \right) \\
 3 p r^2 & r \left( a_3 r^2 - 3 a_1 p^2 \right) & a_1 p^3 - 3 a_2 p^2 r - 3 a_3 p r^2 + a_4 r^3 & a_2 p^3 - 3 a_4 p r^2 + a_5 r^3 \\
 -r^3 & r^2 \left( 3 a_1 p + a_2 r \right) & r \left( a_3 r^2 - 3 a_1 p^2 \right) & a_1 p^3 + a_4 r^3 \\
\end{array}
\right) .
\end{equation*}
\normalsize 
We then find that 
\begin{equation*}
a_1^{3} A^{-1} Q B T^{-1} = \left( N_{\mathcal{O}}^{\left( a_1 p - r \phi_1 \right) } \right)^3 . 
\end{equation*}

Now consider quintic and sextic binary forms one after the other. In these cases we will only exhibit the matrix $Q B T^{-1}$ so that it may be more easily verified that \eqref{siam} holds in these cases. 

When $n = 5$, $Q B T^{-1}$ is given by $\left( Z_1 \ Z_2 \right)$, where 
\footnotesize 
\begin{equation*}
Z_1 = \left(
\begin{array}{ccc}
 p^4 & - a_6 r^4 & 4 a_6 p r^3 \\
 -4 p^3 r & a_1 p^4 - a_5 r^4 & a_2 p^4 + 4 a_5 p r^3 - a_6 r^4 \\
 6 p^2 r^2 & -r \left( 4 a_1 p^3 + a_4 r^3 \right) & a_1 p^4 - 4 a_2 p^3 r + 4 a_4 p r^3 - a_5 r^4 \\
 -4 p r^3 & r^2 \left( 6 a_1 p^2 - a_3 r^2 \right) & - r \left( 4 a_1 p^3 - 6 a_2 p^2 r - 4 a_3 p r^2 + a_4 r^3 \right) \\
 r^4 & - r^3 \left( 4 a_1 p + a_2 r \right) & r^2 \left( 6 a_1 p^2 - a_3 r^2 \right) \\
\end{array}
\right) ,
\end{equation*}
\begin{equation*}
Z_2 = \left(
\begin{array}{cc}
 - 6 a_6 p^2 r^2 & 4 a_6 p^3 r \\
 p \left( a_3 p^3 - 6 a_5 p r^2 + 4 a_6 r^3 \right) & - p^2 \left( - 4 a_5 p r - a_4 p^2 + 6 a_6 r^2 \right) \\
 a_2 p^4 - 4 a_3 p^3 r - 6 a_4 p^2 r^2 + 4 a_5 p r^3 - a_6 r^4 & p \left( a_3 p^3 - 6 a_5 p r^2 + 4 a_6 r^3 \right) \\
 a_1 p^4 - 4 a_2 p^3 r + 4 a_4 p r^3 - a_5 r^4 & a_2 p^4 + 4 a_5 p r^3 - a_6 r^4  \\
 -r \left( 4 a_1 p^3 + a_4 r^3 \right) & a_1 p^4 - a_5 r^4 \\
\end{array}
\right) .
\end{equation*}
\normalsize 
When $n = 6$, $Q B T^{-1}$ is given by $\left( Z_1 \ Z_2 \ Z_3 \right)$, where $Z_1$ , $Z_2$, and $Z_3$ are respectively
\footnotesize 
\begin{equation*}
\left(
\begin{array}{ccc}
 p^5 & a_7 r^5 & - 5 a_7 p r^4  \\
 -5 p^4 r & a_1 p^5 + a_6 r^5 & a_2 p^5 - 5 a_6 p r^4 + a_7 r^5 \\
 10 p^3 r^2 & - r \left( 5 a_1 p^4 - a_5 r^4 \right) & a_1 p^5 - 5 a_2 p^4 r - 5 a_5 p r^4 + a_6 r^5 \\
 -10 p^2 r^3 & r^2 \left( 10 a_1 p^3 + a_4 r^3 \right) & - r \left( 5 a_1 p^4 - 10 a_2 p^3 r + 5 a_4 p r^3 - a_5 r^4 \right) \\
 5 p r^4 & - r^3 \left( 10 a_1 p^2 - a_3 r^2 \right) & r^2 \left( 10 a_1 p^3 - 10 a_2 p^2 r - 5 a_3 p r^2 + a_4 r^3 \right) \\
 -r^5 & r^4 \left( 5 a_1 p + a_2 r \right) & - r^3 \left( 10 a_1 p^2 - a_3 r^2 \right) \\
\end{array}
\right) , 
\end{equation*}
\begin{equation*}
\left(
\begin{array}{c}
 10 a_7 p^2 r^3  \\
 -p \left( - 10 a_6 p r^3 - a_3 p^4 + 5 a_7 r^4 \right) \\
 a_2 p^5 - 5 a_3 p^4 r + 10 a_5 p^2 r^3 - 5 a_6 p r^4 + a_7 r^5 \\
 a_1 p^5 - 5 a_2 p^4 r + 10 a_3 p^3 r^2 + 10 a_4 p^2 r^3 - 5 a_5 p r^4 + a_6 r^5 \\
 -r \left( 5 a_1 p^4 - 10 a_2 p^3 r + 5 a_4 p r^3 - a_5 r^4 \right) \\
 r^2 \left( 10 a_1 p^3 + a_4 r^3 \right) \\
\end{array}
\right) , 
\end{equation*}
\begin{equation*}
\left(
\begin{array}{cc}
 -10 a_7 p^3 r^2 & 5 a_7 p^4 r \\
 p^2 \left( a_4 p^3 - 10 a_6 p r^2 + 10 a_7 r^3 \right) & - p^3 \left( -5 a_6 p r - a_5 p^2 + 10 a_7 r^2 \right) \\
 -p \left( 5 a_4 p^3 r + 10 a_5 p^2 r^2 - 10 a_6 p r^3 - a_3 p^4 + 5 a_7 r^4 \right) & p^2 \left( a_4 p^3 - 10 a_6 p r^2 + 10 a_7 r^3 \right) \\
 a_2 p^5 - 5 a_3 p^4 r + 10 a_5 p^2 r^3 - 5 a_6 p r^4 + a_7 r^5 & - p \left( -10 a_6 p r^3 - a_3 p^4 + 5 a_7 r^4 \right) \\
 a_1 p^5 - 5 a_2 p^4 r - 5 a_5 p r^4 + a_6 r^5 & a_2 p^5 - 5 a_6 p r^4 + a_7 r^5  \\
 -r \left( 5 a_1 p^4 - a_5 r^4 \right) & a_1 p^5 + a_6 r^5 \\
\end{array}
\right) . 
\end{equation*}
\normalsize 
Indeed the calculations required to verify \eqref{siam} are large. However, it is reasonable to believe that there is a proof that it holds in general which uses expressions for the element in the $i$-th row and $j$-th column of each of the matrices $A$, $B$, $Q$, $T$, and $\left( N_{\mathcal{O}}^{(a_1 p -r \phi_1 )} \right)^{n-1}$.

\end{document}